\documentclass[draft,12pt]{amsart}
\usepackage{amscd}

\usepackage{latexsym}
\usepackage{enumerate}

\usepackage{amsfonts}
\usepackage{amssymb}
\usepackage{mathrsfs}

\newcommand{\norm}[1]{\left\lVert#1\right\rVert}

\swapnumbers

\theoremstyle{plain}
\newtheorem{theorem}{Theorem}[section]

\newtheorem{proposition}[theorem]{Proposition}

\theoremstyle{plain}

\newtheorem{definition}[theorem]{Definition}
\newtheorem{remark}[theorem]{Remark}

\def\R{{\rm I\kern-2ptR}}
\def\N{{\rm I\kern-2ptN}}

\def\emptyset{\mathop{\raisebox{.1ex}{$\not
\mathrel{\raisebox{.1ex}{$\scriptstyle\bigcirc$}}$}}}

\def\qed{\hskip .6em \raise1.8pt\hbox{\vrule height4pt
width6pt depth2pt}}
\def\qedd{\hskip .4em \raise1.8pt\hbox{\vrule height3pt
width5pt depth1.8pt}}
\def\sq{\hskip .6em \raise1.8pt\hbox{\vrule
height4pt width6pt depth2pt}}

\def\leaderfill{\leaders\hbox to 1em{\hss.\hss}\hfill}

\begin{document}

\title[]{Sequence-Singular Operators}

\author{Gleb Sirotkin}
\address{Department of Mathematical Sciences, Northern Illinois University, DeKalb, IL 60115}
\email{gsirotkin@niu.edu}

\author{Ben Wallis}
\address{Department of Mathematical Sciences, Northern Illinois University, DeKalb, IL 60115}
\email{bwallis@niu.edu}

\thanks{Mathematics Subject Classification: 46B06, 46B25, 46B45, 47L10, 47L20\\\indent Keywords:  functional analysis, Banach spaces, operator ideals}

\begin{abstract} In this paper we study two types of collections of operators on a Banach space on the subject of forming operator ideals. One of the types allows us to construct an uncountable chain of closed ideals in each of the operator algebras $\mathcal{L}(\ell_1\oplus\ell_q)$, $1<q<\infty$, and $\mathcal{L}(\ell_1\oplus c_0)$.  This finishes answering a longstanding question of Pietsch.\end{abstract}

\maketitle
\theoremstyle{plain}
%%%%%%%%%%%%%%%%%%%%%%%%%%%%%%%%%%%%%%%%%%%%%%Introduction
\section{Introduction}

Fix a seminormalized basis $e=(e_n)$ for a Banach space $E$.  Following Beanland and Freeman \cite{BF11}, we say that an operator $T\in\mathcal{L}(X,Y)$, $X$ and $Y$ Banach spaces, is {\bf $\boldsymbol{(e_n)}$-singular} just in case for every normalized basic sequence $(x_n)$ in $X$, the image sequence $(Tx_n)$ fails to dominate $(e_n)$.  We denote by $\mathcal{WS}_{e,\omega_1}(X,Y)$ the class of all $(e_n)$-singular operators in $\mathcal{L}(X,Y)$.  In \cite[Proposition 2.8]{BF11} the following interesting results were proved about class $\mathcal{WS}_{e,\omega_1}$ for certain nice choices of $e$.
\begin{itemize}\item  If $e=(e_n)$ denotes the canonical basis for $c_0$ then $\mathcal{WS}_{e,\omega_1}=\mathcal{K}$, the compact operators.
\item  If $e=(e_n)$ denotes the summing basis for $c_0$ then $\mathcal{WS}_{e,\omega_1}=\mathcal{W}$, the weakly compact operators.
\item  If $e=(e_n)$ denotes the canonical basis for $\ell_1$ then $\mathcal{WS}_{e,\omega_1}=\mathcal{R}$, the Rosenthal operators.\end{itemize}
\noindent (Recall that an operator $T\in\mathcal{L}(X,Y)$ is {\bf Rosenthal} just in case for every bounded sequence $(x_n)$ in $X$, $(Tx_n)$ admits a weak Cauchy subsequence.)  Each of these classes is a norm-closed operator ideal, and so it is natural to conjecture that class $\mathcal{WS}_{e,\omega_1}$ could also form an operator ideal for other nice choices of $e=(e_n)$.  In particular, we might expect $\mathcal{WS}_{e,\omega_1}$ to be an operator ideal whenever $e=(e_n)$ is the canonical basis of $\ell_p$, $1<p<\infty$.  

In the present paper, we show that the above conjecture is false, and indeed that for any $1<p<\infty$ we can always choose spaces $X$ and $Y$ such that $\mathcal{WS}_{e,\omega_1}(X,Y)$ fails to be closed under addition when $e=(e_n)$ is the canonical basis for $\ell_p$, $1<p<\infty$.

Despite this, we might still be able to use them, or variants thereof, to investigate the closed ideal structure of the operator algebra $\mathcal{L}(X)$ for certain choices of $X$.  Indeed, using a descriptive set-theoretic result from \cite{BF11} together with a modification of the definition of $\mathcal{WS}_{e,\xi}$ to show that $\mathcal{L}(\ell_1\oplus\ell_q)$, $1<q<\infty$, and $\mathcal{L}(\ell_1\oplus c_0)$ each admit an uncountable chain of closed ideals.  This is especially significant since it represents the last ingredient needed to answer a longstanding open question of Pietsch (\cite[Problem 5.33]{Pi78}).

For the most part, all definitions and notation are standard, as are found, for instance, in \cite{AK06}.  However, we will restate some of the most important such here.  Let $\mathcal{J}$ be a subclass of the class $\mathcal{L}$ of all continuous linear operators between Banach spaces, and if $X$ and $Y$ are Banach spaces then we write $\mathcal{J}(X,Y)=\mathcal{L}(X,Y)\cap\mathcal{J}$, a component.  We say that $\mathcal{J}$ has the {\bf ideal property} whenever $BTA\in\mathcal{J}(W,Z)$ for all $A\in\mathcal{L}(W,X)$, $B\in\mathcal{L}(Y,Z)$, and $T\in\mathcal{J}(X,Y)$, and all Banach spaces $W$, $X$, $Y$, and $Z$.  If in addition every component $\mathcal{J}(X,Y)$ is a linear subspace of $\mathcal{L}(X,Y)$ containing all the finite-rank operators therein, then $\mathcal{J}$ is an \textbf{operator ideal}.  We say that $\mathcal{J}$ is norm-closed (closed under addition) whenever all its components $\mathcal{J}(X,Y)$ are norm-closed (closed under addition) in $\mathcal{L}(X,Y)$.  Let us also borrow a piece of terminology from \cite{Sc12}:  If $X$ and $Y$ are Banach spaces, then a linear subspace $\mathcal{J}$ of $\mathcal{L}(X,Y)$ is called a {\bf subideal} just in case whenever $A\in\mathcal{L}(X)$, $B\in\mathcal{L}(Y)$, and $T\in\mathcal{J}$, we have $BTA\in\mathcal{J}$.  A subideal of $\mathcal{L}(X)$ is called, simply, an {\bf ideal}.  (For operator algebras, this coincides with the notion of an {\it ideal} in the algebraic sense.)

If $M$ is an infinite subset of $\mathbb{N}$, then denote by $[M]$ the family of all infinite subsets of $M$, and denote by $[M]^{<\omega}$ the family of all finite subsets of $M$.  For $n\in\mathbb{N}$ let $[M]^{\leq n}=\{A\in[M]^{<\omega}:\#A\leq n\}$, i.e. the family of all subsets of $M$ of size $\leq n$.  If $\mathcal{F}$ is a subset of $[\mathbb{N}]^{<\omega}$ and $M=(m_i)\in[\mathbb{N}]$, then we define
\[\mathcal{F}(M)=\{(m_i)_{i\in E}:E\in\mathcal{F}\}.\]
If $\mathcal{F}$ and $\mathcal{G}$ are both subsets of $[\mathbb{N}]^{<\omega}$ then we define
\[\mathcal{F}[\mathcal{G}]=\left\{\bigcup_{i=1}^nE_i:E_1<\cdots<E_n,E_i\in\mathcal{G}\forall i,(\min E_i)_{i=1}^n\in\mathcal{F}\right\}.\]

Let us now define the \textbf{Schreier families}.  These are denoted $\mathcal{S}_\xi$ for each countable ordinal $0\leq\xi<\omega_1$, and we must define them as follows.  For $\xi = 0$, we put $\mathcal{S}_0:=\{\{n\}:n\in\mathbb{N}\}\cup\{\emptyset\}$.  For the case $\xi=\zeta+1$ for a countable ordinal $1\leq\zeta<\omega_1$, we define $\mathcal{S}_\xi$ as the set containing $\emptyset$ together with all $F\subset\mathbb{N}$
such that there exist $n\in\mathbb{N}$ and a decomposition $F=\bigcup_{k=1}^nF_k$ with sets $F_k$ from $\mathcal{S}_\zeta$ satisfying $n\le F_1<\cdots<F_n$. In case $\xi$ is a limit ordinal we fix a strictly increasing sequence $(\zeta_n)$ of non-limit-ordinals satisfying $\sup_n\zeta_n=\xi$, and define $\mathcal{S}_\xi:=\bigcup_{n=1}^\infty\{F\in\mathcal{S}_{\zeta_n}:n\leq F\}$.

For convenience, in some contexts we will write $\mathcal{S}_{\omega_1}$ for the family of all finite subsets of $\mathbb{N}$.  In other words, we have $\mathcal{S}_{\omega_1}$ being identical to $[\mathbb{N}]^{<\omega}$, and which notation we use will depend on the context.  This is, admittedly, somewhat of an abuse of notation, as there is no such thing as the ``$\omega_1$st Schreier family,'' but it will greatly simplify the writing.

It is well-known that the Schreier families (including $\mathcal{S}_{\omega_1}$) have the {\bf spreading property}, which is to say that if $\{n_1<\cdots<n_k\}\in\mathcal{S}_\xi$ and $\{m_1<\cdots<m_k\}\subseteq\mathbb{N}$ satisfies $n_j\leq m_j$ for all $j=1,\cdots,k$, then $\{m_1<\cdots<m_k\}\in\mathcal{S}_\xi$.  Please note that, in general, it is not true that $\mathcal{S}_\xi\subseteq\mathcal{S}_\zeta$ for all $1\leq\xi\leq\zeta\leq\omega_1$.  However, we do always have $\mathcal{S}_1\subseteq\mathcal{S}_\xi$, $1\leq\xi\leq\omega_1$.  Furthermore, given any $1\leq\xi\leq\zeta\leq\omega_1$ there exist $d=d(\xi,\zeta)$ such that if $S\in\mathcal{S}_\xi$ with $d\leq\min S$ then $S\in\mathcal{S}_\zeta$.  (That $\mathcal{S}_1\subseteq\mathcal{S}_\xi$ is obvious, and the remaining facts can all be found, for instance, in \cite[III.2,p1051]{AGR03}.)

We will also need a pair of technical results regarding the Schreier families.

\begin{proposition}[{\cite[Lemma 1.3]{Po09}}]\label{2x}  Let $1\leq\xi\leq\omega_1$ and $A\in\mathcal{S}_\xi$.  Then
\[\{2n,2n+2:n\in A\}\in\mathcal{S}_\xi.\]\end{proposition}

\begin{proposition}[{\cite[Proposition 3.2(b)]{OTW97}}]\label{1.iii}  If $1\leq\xi,\zeta<\omega_1$ then there is $L\in[\mathbb{N}]$ such that $\mathcal{S}_\xi[\mathcal{S}_\zeta](L)\subseteq\mathcal{S}_{\zeta+\xi}$.\end{proposition}

For a fixed ordinal $1\leq\xi\leq\omega_1$, we say that a sequence $(x_n)$ in a Banach space $X$ {\bf $\boldsymbol{\mathcal{S}_\xi}$-dominates} another sequence $(y_n)$ in a Banach space $Y$ just in case there is a constant $C\in[1,\infty)$ satisfying
\[\norm{\sum_{n\in F} a_ny_n}\leq C\norm{\sum_{n\in F} a_nx_n}\]
for all $(a_n)\in c_{00}$ and $F\in\mathcal{S}_\xi$. When $\xi = \omega_1$ we will simply
say that $(x_n)$ {\bf dominates} $(y_n)$ and this will coincide with the usual notion of domination in literature.
In this case we write $(x_n)\geq_C(y_n)$ or, when $C$ is unimportant, simply $(x_n)\geq(y_n)$. 

The remainder of the paper is divided into two parts.  In the next section we use Lorentz sequence spaces to show that class $\mathcal{WS}_{e,\xi}$, $1\leq\xi\leq\omega_1$, fails to be closed under addition when $e=(e_n)$ is chosen from among the canonical bases for $\ell_p$, $1<p<\infty$.  After that, in the last section, we define and study new classes $\mathcal{JS}_{e,\xi}$. There we show that, in the case when $e=(e_n)$ is the canonical basis for either $\ell_p$, $1\leq p<\infty$, or $c_0$, $\mathcal{JS}_{e,\omega_1}$ is a norm-closed operator ideal.  We conclude by using the new classes $\mathcal{JS}_{e,\xi}$ to show in Theorem \ref{main} that $\mathcal{L}(X)$ admits infinitely many closed ideals whenever $X$ contains complemented copies of $\ell_1$ and either some $\ell_q$, $1<q\leq\infty$, or $c_0$; the same holds if $X$ instead contains complemented copies of $\ell_\infty$ and some $\ell_p$, $1\leq p<\infty$.

\section{Classes $\mathcal{WS}_{e,\xi}$}

\begin{definition}Let $X$ and $Y$ be Banach spaces, and let $1\leq\xi\leq\omega_1$ be an ordinal.  Fix any normalized basis $e=(e_n)$.  We define $\mathcal{WS}_{e,\xi}(X,Y)$ as the set of all operators $T\in\mathcal{L}(X,Y)$ such that for any normalized basic sequence $(x_n)$ in $X$, the image sequence $(Tx_n)$ fails to $\mathcal{S}_\xi$-dominate $(e_n)$.\end{definition}

The following is a straightforward observation based on the basic constant of $(e_n)$.

\begin{proposition}\label{WS-compact}Let $1\leq\xi\leq\omega_1$, let $e=(e_n)$ be any normalized basis, and let $X$ and $Y$ be Banach spaces with $T\in\mathcal{L}(X,Y)$.  If $(x_n)$ is a sequence in $X$ and $(Tx_n)$ has a norm-convergent subsequence, then for every $\epsilon>0$ and $N>0$ there exist $m$ and $n$ such that $N<n<m$ and
\[\norm{Tx_m-Tx_n}_Y<\epsilon\norm{e_m-e_n}_E.\]
If this happens for every normalized basic sequence in $X$, then $T\in\mathcal{WS}_{e,\xi}(X,Y)$.  In particular, we always have $\mathcal{K}(X,Y)\subseteq\mathcal{WS}_{e,\xi}(X,Y)$.\end{proposition}

In view of the last proposition, in several proofs, we will be concentrating on sequences with no convergent subsequences. Applying Rosenthal's $\ell_1$ Theorem (cf., e.g., \cite[Theorem 10.2.1]{AK06}) together with \cite[Theorem 1.5.4]{AK06}, we obtain the following standard fact.

\begin{proposition}\label{difference-sequence}Suppose $X$ and $Y$ are Banach spaces with $T\in\mathcal{L}(X,Y)$.  If $(x_n)$ is a bounded sequence in $X$, then there exists a subsequence $(x_{n_k})$ so that exactly one of the following holds.
\begin{itemize}\item $(Tx_{n_k})$ is norm-convergent.
\item The sequences $(x_{n_{4k+2}}-x_{n_{4k}})$ and $(Tx_{n_{4k+2}}-Tx_{n_{4k}})$ are both seminormalized and basic, and each of them are either weakly null or else equivalent to the canonical basis of $\ell_1$.\end{itemize}\end{proposition}

\noindent This gives us an equivalent characterization of the classes in terms of bounded sequences instead of normalized basic sequences.

\begin{proposition}\label{seminormalized}Let $X$ and $Y$ be Banach spaces, and let $1\leq\xi\leq\omega_1$ be an ordinal.  Fix any normalized basis $e=(e_n)$ satisfying $(e_{4n+2}-e_{4n})\geq(e_{n})$.  Then $\mathcal{WS}_{e,\xi}(X,Y)$ is the set of all operators $T\in\mathcal{L}(X,Y)$ such that for any bounded sequence $(x_n)$ in $X$, the image sequence $(Tx_n)$ fails to $\mathcal{S}_\xi$-dominate $(e_n)$.\end{proposition}

\begin{proof}Suppose $T$ is in $\mathcal{WS}_{e,\xi}(X,Y)$ and fix any bounded sequence $(x_n)$. If $(Tx_n)$ has a norm-convergent subsequence then we are done by Proposition \ref{WS-compact}.  Otherwise, by Proposition \ref{difference-sequence} we can find a subsequence $(x_{n_k})$ so that $(z_k)$ and $(Tz_k)$ are each seminormalized basic, where we define
\[z_k:=x_{n_{4k+2}}-x_{n_{4k}}.\]
Fix $\epsilon>0$, and pass to a further subsequence if necessary so that $\norm{Tz_k}_Y\to r$ for some $0<r<\infty$, and quickly enough so that by the Principle of Small Perturbations, $(\frac{1}{r}e_k)$ is $C$-equivalent, $C\geq 1$, to $(\frac{1}{\norm{z_{k}}_X}e_k)$.  Notice that $(\frac{z_{k}}{\norm{z_{k}}_X})$ is a normalized basic sequence in $X$.  Since $T\in\mathcal{WS}_{e,\xi}(X,Y)$, we can therefore find $(a_n)\in c_{00}$ with support in $\mathcal{S}_\xi$ such that
\[\norm{\sum a_kT\frac{z_{k}}{\norm{z_{k}}_X}}_Y<\frac{\epsilon}{Cr}\norm{\sum a_ke_k}_E\leq\epsilon\norm{\sum\frac{a_k}{\norm{z_{k}}_X}e_k}_E.\]
This is sufficient by Proposition \ref{2x} together with the spreading property of $\mathcal{S}_\xi$ and fact that $(e_n)\leq(e_{4n+2}-e_{4n})$.\end{proof}

For the remainder of this section, let us develop the machinery required to show that class $\mathcal{WS}_{e,\xi}$ fails to be closed under addition whenever $(e_n)$ is the canonical basis for $\ell_p$, $1<p<\infty$.

\begin{proposition}\label{WS-direct-sum}Suppose $(e_n)$ is a normalized basic sequence satisfying $(e_{2n+2}-e_{2n})\geq(e_{n})$, and fix $1\leq\xi\leq\omega_1$.  Let $X_1$ and $X_2$ be Banach spaces, $T$ be an operator from $\mathcal{WS}_{e,\xi}(X_1,X_2)$, and let $Y_1$ and $Y_2$ denote two more Banach spaces.  Then
\[T\oplus 0\in\mathcal{WS}_{e,\xi}(X_1\oplus Y_1,X_2\oplus Y_2)\;\;\text{ and }\;\;0\oplus T\in\mathcal{WS}_{e,\xi}(Y_1\oplus X_1,Y_2\oplus X_2).\]\end{proposition}

\begin{proof}By symmetry it suffices to prove the first statement.  Let $(x_n\oplus y_n)$ be a bounded sequence in $X_1\oplus Y_1$, and let $\epsilon>0$.  Then $(x_n)$ is bounded in $X_1$, and so by Proposition \ref{seminormalized} we can find $(a_k)\in c_{00}$ with support in $\mathcal{S}_\xi$ such that
\[\norm{\sum a_n(T\oplus 0)(x_n\oplus y_n)}_{X_2\oplus Y_2}=\norm{\sum a_nTx_n}_{X_2}<\epsilon\norm{\sum a_ne_n}_E.\]\end{proof}

Let us now describe the Lorentz sequence spaces.  Suppose $1\leq q<\infty$, and let $w=(w_n)_{n=1}^\infty\in c_0\setminus\ell_1$ be a nonincreasing sequence with $w_1 =1$.  Denote by $\Pi$ the set of all permutations of $\mathbb{Z}^+$.  We define the set $d(w,q)$ as the set of all scalar sequences $(a_n)_{n=1}^\infty$ such that
\[\norm{(a_n)_{n=1}^\infty}_{d(w,q)}:=\sup_{\sigma\in\Pi}\left(\sum_{n=1}^{\infty}|a_{\sigma(n)}|^qw_n\right)^{1\over q} < \infty.\]
When endowed with the norm $\norm{\cdot}_{d(w,q)}$, the set $d(w,q)$ defines a Banach space with a canonical basis which is normalized and symmetric.  Note that, due to the properties of $w$, we can equivalently characterize the $d(w,q)$ norm as follows.  If $(a_n)_{n=1}^\infty\in d(w,q)$, then we denote by $(a_n^*)_{n=1}^\infty$ the nonincreasing rearrangement of $(|a_n|)_{n=1}^\infty$.  In this case,
\[\norm{(a_n)_{n=1}^\infty}_{d(w,q)}=\norm{(a_n^*w_n^{1/q})_{n=1}^\infty}_{\ell_q}.\]
We shall call any such space $d(w,q)$ a {\bf Lorentz sequence space}.  See \cite[\S4.e]{LT77} for further discussion of these spaces.

\begin{proposition}\label{lorentz-dominates-lp}Fix numbers $p$ and $q$ such that $1\leq q<p<\infty$.  Let $w=(w_n)$ and $\tilde{w}=(\tilde{w}_n)$ be nonincreasing sequences, $w_1=\tilde{w}_1=1$, lying in $c_0\setminus\ell_1$, with $(x_n)$ and $(\tilde{x}_n)$ the respective canonical bases for $d(w,q)$ and $d(\tilde{w},q)$.  Suppose that $w_n+\tilde{w}_n\geq n^{\frac{1}{p}-\frac{1}{q}}$.  Then $(x_n\oplus\tilde{x}_n)\subset X$ 1-dominates the canonical basis of $\ell_p$, where we define
\[X:=\left(d(w,q)\oplus d(\tilde{w},q)\right)_{\ell_q}.\]\end{proposition}

\begin{proof}Set, for convenience, $v_n:=w_n+\tilde{w}_n$.  Our goal is to prove that for each $k$ and every $(a_n)\in c_{00}$ with $\text{supp}(a_n) \subset [1, k]$ we have the following inequality:
\[\norm{\sum a_n(x_n\oplus\tilde{x}_n)}^q_X=\norm{\sum_{n=1}^k a_n(x_n\oplus\tilde{x}_n)}^q_X=\sum_{n=1}^k a_n^{*q}v_n\geq\norm{(a_n)}^q_{\ell_p}.\]
Clearly, the result holds when $k=1$.  Now assume it holds for some $k\in\mathbb{N}$ and let us prove the inequality for $k+1$. Using the inductive assumption, we can write 
\[\sum_{n=1}^{k+1} a_n^{*q}v_n\ge \norm{(a_n^*)_{n=1}^k}^q_{\ell_p}+a_{k+1}^{*q}v_{k+1}\ge \norm{(a_n^*)_{n=1}^k}^q_{\ell_p}+a_{k+1}^{*q}(k+1)^{\frac{1}{p}-\frac{1}{q}}\]
Define a function $f:[0,a_k^*]\to\mathbb{R}$ by the rule
\[f(t)=\norm{(a_n^*)_{n=1}^k}_{\ell_p}^q+t^q(k+1)^{\frac{1}{p}-\frac{1}{q}}-\left(\sum_{n=1}^ka_n^{*p}+t^p\right)^{q\over p}.\]
Notice that if we show that $f$ is nonnegative, then we are done. For that, observe
\begin{eqnarray*}
f'(t)=qt^{q-1}(k+1)^{\frac{1}{p}-\frac{1}{q}}-qt^{p-1}\left(\sum_{n=1}^ka_n^{*p}+t^p\right)^{{q\over p}-1}
\\=qt^{q-1}\left[(k+1)^{\frac{1}{p}-\frac{1}{q}}-t^{p-q}\left(\sum_{n=1}^ka_n^{*p}+t^p\right)^{q-p \over p}\right]
\quad\text{ then, as each }a_n^*\ge a_k^*\ge t,
\\\ge qt^{q-1}\left[(k+1)^{\frac{1}{p}-\frac{1}{q}}-t^{p-q}\left((k+1)t^p\right)^{q-p \over p}\right]
=qt^{q-1}\left[(k+1)^{q-p \over pq}-(k+1)^{q-p \over p}\right].
\end{eqnarray*}
The latter is nonnegative for each $t\in [0,a_k^*]$ due to $1\le q< p$. 
Since $f(0)=0$ the proposition is proved.
\end{proof}

\begin{proposition}\label{alternating-pair}Let $1\leq q<p<\infty$.  There exists a pair of nonincreasing sequences $w=(w_n)$ and $\tilde{w}=(\tilde{w}_n)$, $w_1=\tilde{w}_1=1$, lying in $c_0\setminus\ell_1$, such that neither of the respective canonical bases $(x_n)$ or $(\tilde{x}_n)$ of $d(w,q)$ and $d(\tilde{w},q)$ dominates the canonical basis of $\ell_p$, but their $\ell_q$-direct sum $(x_n\oplus\tilde{x}_n)$ does.\end{proposition}

\begin{proof}Put $s:=\frac{1}{q}-\frac{1}{p}$, and as usual, let $(s_n)\subset c_{00}$ denote the summing basis, i.e. the basis defined by
\[s_n=(\underbrace{1,1,\cdots,1}_{n\text{ times}},0,0,0,\cdots).\]

\noindent We will construct sequences $w$ and $\tilde{w}$ together with a sequence of indices
\[0=m_1<n_1<m_2<n_2<m_3<n_3<\cdots\]
inductively. For each $n$ we will have either $w_n = {1 \over n^{s}}$ or else $\tilde{w}_n= {1 \over n^{s}}$.  By Proposition \ref{lorentz-dominates-lp}, this will guarantee that $(x_n\oplus{x}_n)$ dominates the canonical basis of $\ell_p$.  Second, we will have $\norm{s_{n_k}}_{d(\tilde{w},q)}\leq k^{-1}\norm{s_{n_k}}_{\ell_p}$
 if $k$ is odd and $\norm{s_{n_k}}_{d(w,q)}\leq k^{-1}\norm{s_{n_k}}_{\ell_p}$ otherwise.
Thus, neither $(x_n)$ nor $(\tilde{x}_n)$ will dominate the canonical basis of $\ell_p$.

Begin by defining $n_1=1$ and $w_1=\tilde{w}_1=1$.  Notice that this means $\norm{s_{n_1}}_{d(\tilde{w},q)}\leq\norm{s_{n_1}}_{\ell_p}$ and $w_{n_1}= n_1^{-s}$.

For the inductive step, suppose $m_1<n_1<\cdots<m_k<n_k$, $(w_n)_{n=1}^{n_k}$, and $(\tilde{w}_n)_{n=1}^{n_k}$ have all been defined, and consider the case where $k+1$ is even.  Assume, in addition, that $w_{n_k}=n_k^{-s}$. Set $m_{k+1}>n_k$ to be any number such that $\tilde{w}_{n_k}\geq(m_{k+1}+1)^{-s}$.  Define $w_j = {1\over j^s}$ and $\tilde{w}_{j}=\tilde{w}_{n_k}$ for all $n_k<j\le m_{k+1}$.  
%Also define $w_{n_k+j}=(n_k+j)^{-s}$ for all $j=1,\cdots,m_{k+1}-n_k$.  
Due to the inductive hypothesis, we have $w_{n_k+1}<n_k^{-s}\leq w_{n_k}$ so that both sequences are nonincreasing as required. Also, by the choice of $m_{k+1}$, we have $\tilde{w}_{m_{k+1}}\geq(m_{k+1}+1)^{-s}$.  

Next, let $n_{k+1}>m_{k+1}$ be such that
\[\norm{s_{m_k}}_{d(w,q)}+1\leq{1\over k+1}\norm{s_{n_{k+1}}}_{\ell_p}.\]
Define $\tilde{w}_{j}={1 \over j^s}$ and $w_{j}={1 \over 2^{j}}w_{m_{k+1}}$ for all $m_{k+1}<j \le n_{k+1}$.  Then
\begin{multline*}\norm{s_{n_{k+1}}}_{d(w,q)}\leq\norm{s_{m_{k+1}}}_{d(w,q)}+\left(\sum_{j\ge 1}{w_{m_{k+1}}\over 2^j}\right)^{1\over q}
\leq\norm{s_{m_{k+1}}}_{d(w,q)}+1\leq{\norm{s_{n_{k+1}}}_{\ell_p}\over k+1}.\end{multline*}
The case where $k+1$ is odd we handle in a similar fashion.\end{proof}

We also need the following result from \cite{KPSTT12}.

\begin{proposition}[\cite{KPSTT12}, Lemma 4.10]\label{4.10}Let $1\leq q<\infty$ and $w=(w_n)\in c_0\setminus\ell_1$ be nonincreasing with $w_1=1$.  Let $I_{q,w}:\ell_q\to d(w,q)$ denote the formal identity between canonical bases.  Suppose $(x_n)$ is a seminormalized block basic sequence in $\ell_q$.  If $(I_{q,w}x_n)$ is seminormalized in $d(w,q)$ then it admits a subsequence equivalent to the canonical basis of $d(w,q)$.\end{proposition}

\begin{proposition}\label{identity-is-WS}Let $1<q<p<\infty$ and $1\leq\xi\leq\omega_1$, and suppose $(e_n)$ is the canonical basis of $\ell_p$.  Let $d(w,q)$ be a Lorentz sequence space whose canonical basis fails to dominate the canonical basis of $\ell_p$.  Then the formal identity $I_{q,w}:\ell_q\to d(w,q)$ is class $\mathcal{WS}_{e,\xi}$\end{proposition}

\begin{proof}Fix any $\epsilon>0$ and consider a normalized basic sequence $(x_n)_{n=1}^\infty$ in $\ell_q$.  Since every seminormalized basic sequence in a reflexive space is weakly null (as being shrinking), by the Bessaga-Pe\l czy\'{n}ski Selection Principle, we can find a subsequence $(x_{n_k})$ and successive finite subsets of $\mathbb{N}$ which we denote
\[E_1<E_2<E_3<\cdots\]
so that $(E_kx_{n_k})$ is a seminormalized block basic sequence in $\ell_q$, and
\[z_k:=x_{n_k}-E_kx_{n_k}\]
satisfies $\norm{z_k}_{\ell_q}\to 0$. By Proposition \ref{WS-compact}, we may assume that the $d(w,q)$-block sequence $(I_{q,w}E_kx_{n_k})_{k=1}^\infty$ is seminormalized. Thus, we can pass to a further subsequence if necessary so that by Proposition \ref{4.10}, $(I_{q,w}E_kx_{n_k})_{k=1}^\infty$ is $K$-equivalent to the canonical basis of $d(w,q)$, where $K\geq 1$.  Pass to a still further subsequence so that $\norm{I_{q,w}z_k}_{d(w,q)}\leq 2^{-k-1}\epsilon$.  Now, since the canonical basis of $d(w,q)$ fails to dominate the canonical basis of $\ell_p$, we can find $(c_n)\in c_{00}$ such that
\[\norm{(c_k)}_{d(w,q)}<\frac{\epsilon}{2K}\norm{(c_k)}_{\ell_p}.\]
By the symmetric property of the canonical basis of $d(w,q)$ we can assume $\text{supp}(c_k)\in\mathcal{S}_1\subseteq\mathcal{S}_\xi$.  Then
\begin{multline*}\norm{\sum c_kI_{q,w}x_{n_k}}_{d(w,q)}\leq\norm{\sum c_kI_{q,w}E_kx_{n_k}}_{d(w,q)}+\sum\norm{c_kI_{q,w}z_k}_{d(w,q)}\\\\\leq K\norm{(c_k)}_{d(w,q)}+\frac{\epsilon}{2}\norm{(c_k)}_{\ell_\infty}<\frac{\epsilon}{2}\norm{(c_k)}_{\ell_p}+\frac{\epsilon}{2}\norm{(c_k)}_{\ell_\infty}\leq\epsilon\norm{(c_k)}_{\ell_p}.\end{multline*}\end{proof}

We are now ready to prove the main result of this section.

\begin{theorem}\label{WS-not-closed-under-addition}Let $e=(e_n)$ denote the canonical basis of $\ell_p$, $1<p<\infty$, and let $1\leq\xi\leq\omega_1$ be an ordinal.  Then class $\mathcal{WS}_{e,\xi}$ fails to be closed under addition, and hence is not an operator ideal.\end{theorem}

\begin{proof}Let $q\in(1,p)$ and choose $w$ and $\tilde{w}$ as in Proposition \ref{alternating-pair} so that neither of the respective canonical bases $(x_n)$ or $(\tilde{x}_n)$ of $d(w,q)$ and $d(\tilde{w},q)$ dominates the canonical basis of $\ell_q$, but their $\ell_q$-direct sum $(x_n\oplus\tilde{x}_n)$ does.  Let $I_{q,w}:\ell_q\to d(w,q)$ and $I_{q,\tilde{w}}:\ell_q\to d(\tilde{w},q)$ denote the formal identity operators.  Then by Propositions \ref{WS-direct-sum} and \ref{identity-is-WS},
\[I_{q,w}\oplus 0:\ell_q\oplus\ell_q\to d(w,q)\oplus d(\tilde{w},q)\;\;\;\text{ and }\;\;\;0\oplus I_{q,\tilde{w}}:\ell_q\oplus\ell_q\to d(w,q)\oplus d(\tilde{w},q)\]
are both class $\mathcal{WS}_{e,\xi}$.  However, since $(x_n\oplus\tilde{x}_n)$ dominates the canonical basis of $\ell_p$, their sum
\[I_{q,w}\oplus 0 + 0\oplus I_{q,\tilde{w}}=I_{q,w}\oplus I_{q,\tilde{w}}\]
is not class $\mathcal{WS}_{e,\xi}$.\end{proof}

\section{Closed ideals in $\mathcal{L}(\ell_1\oplus\ell_q)$, $1<q<\infty$, and $\mathcal{L}(\ell_1\oplus c_0)$}

For many years, researchers have been interested in discovering whether or not, given a particular Banach space $X$, the operator algebra $\mathcal{L}(X)$ admits infinitely many closed ideals.  In the case of many classical Banach spaces, this has long been decided.  For instance, in 1960 it was shown that $\mathcal{L}(\ell_p)$, $1\leq p<\infty$, and $\mathcal{L}(c_0)$ admit exactly three closed ideals (\cite{GMF67}).  This also took care of the case $\mathcal{L}(L_2)$, since $L_2\cong\ell_2$.  By 1978 it was discovered that $\mathcal{L}(L_p)$ admits infinitely many closed ideals for $p\in(1,2)\cup(2,\infty)$ (\cite[Theorem 5.3.9]{Pi78}), and in 2015 this was improved to show continuum many (\cite[Theorem 1.1]{Wa15}).  Also in 1978 was shown that $\mathcal{L}(C[0,1])$ admits uncountably many closed ideals (\cite[Theorem 5.3.11]{Pi78}).  Whether $\mathcal{L}(L_1)$ and $\mathcal{L}(L_\infty)\cong\mathcal{L}(\ell_\infty)$ admit infinitely many closed ideals remains a significant open question.

Besides these classical cases, the closed ideal structures of $\mathcal{L}(\ell_p\oplus\ell_q)$, $1\leq p<q<\infty$, have generated a great deal of interest.  Although Pietsch asked as early as 1978 whether these operator algebras admit infinitely many closed ideals (\cite[Problem 5.33]{Pi78}), the question remained entirely open for over 36 years.  Indeed, not until 2014 was it finally shown that $\mathcal{L}(\ell_p\oplus\ell_q)$ admits continuum many closed ideals whenever $1<p<q<\infty$ (\cite{SZ14}).  Then, in 2015 was shown that this result extends to $\mathcal{L}(\ell_p\oplus c_0)$ and $\mathcal{L}(\ell_1\oplus\ell_q)$ in the special cases $1<p<2<q<\infty$ (\cite[Theorem 1.1]{Wa15}).  

In this section we close Pietsch's question by proving the following.

\begin{theorem}\label{main}Suppose $X$ is a Banach space containing a complemented copy of $\ell_1$, and a complemented copy either of $\ell_q$, $1<q\leq\infty$, or of $c_0$.  Then $\mathcal{L}(X)$ admits an uncountable chain of closed ideals.  The same is true if $X$ contains a complemented copy of $\ell_p$, $1\leq p<\infty$, and of $\ell_\infty$.  In particular, $\mathcal{L}(\ell_1\oplus\ell_q)$, $1<q\leq\infty$, and $\mathcal{L}(\ell_1\oplus c_0)$ each admit an uncountable chain of closed ideals.\end{theorem}

\noindent Unfortunately, the cases $\mathcal{L}(\ell_p\oplus c_0)$ fail to dualize, and remain open for $2\leq p<\infty$.

Note that in addition to the above operator algebras, we can also close some of the remaining cases for Rosenthal's $X_p$ spaces and Woo's $X_{p,r}$ generalizations thereof.  Let us take a moment to recall the definitions of these spaces.  Pick any $1\leq r<p\leq\infty$.  Let $(f_n)$ and $(g_n)$ denote the respective canonical bases of $\ell_p$ (or $c_0$ if $p=\infty$) and $\ell_r$.  Let $(w_n)\in c_0$ be any sequence of positive numbers tending to zero, and satisfying the condition that $\sum w_n^{pr/(p-r)}=\infty$ (or $\sum w_n^r=\infty$ if $p=\infty$).  Set $e_n=f_n\oplus_\infty w_ng_n$, vectors lying in $\ell_p\oplus_\infty\ell_r$ (or $c_0\oplus_\infty\ell_r$ if $p=\infty$).  Then we can define Woo's spaces $X_{p,r}=[e_n]$.  Rosenthal's spaces are just special cases of Woo's spaces, and we may define them for any $p\in[1,2)\cup(2,\infty]$ by setting $X_p=X_{p,2}$ if $p\in(2,\infty]$ and $X_p=X_{p',2}^*$, $\frac{1}{p}+\frac{1}{p'}=1$, if $p\in[1,2)$.  It had previously been observed that by \cite[Theorem 1.1]{Wa15} the operator algebra $\mathcal{L}(X_{p,r})$ admits continuum many closed ideals whenever $1<r<p<\infty$, or whenever $p=\infty$ and $r\in(1,2)$.  Indeed, so do their dual space algebras $\mathcal{L}(X_{p,r}^*)$, for the same choices of $r$ and $p$.  Note that this also gives continuum many closed ideals in $\mathcal{L}(X_p)$ and $\mathcal{L}(X_p^*)$ whenever $p\in(1,2)\cup(2,\infty)$.  However, due to the fact that $X_{p,r}$ always contains complemented copies of $\ell_p$ (or $c_0$, if $p=\infty$) and $\ell_r$ (cf. \cite[Corollary 3.2]{Woo75}), by Theorem \ref{main} we now have uncountably many closed ideals in $\mathcal{L}(X_{p,1})$ and $\mathcal{L}(X_{p,1}^*)$ for all choices of $1<p\leq\infty$, and $\mathcal{L}(X_{\infty,r}^*)$ for all $1\leq r<\infty$.  So do $\mathcal{L}(X_{\infty,1})$, $\mathcal{L}(X_1)$, $\mathcal{L}(X_1^*)$, and $\mathcal{L}(X_\infty^*)$.  Among Woo's and Rosenthal's spaces and their duals, this leaves open only the cases $\mathcal{L}(X_{\infty,r})$, $r\in[2,\infty)$, and $\mathcal{L}(X_\infty)$.

We will prove Theorem \ref{main} by modifying the definition of classes $\mathcal{WS}_{e,\xi}$, and using them to produce uncountable chains of closed subideals in certain operator algebras.  The new classes are as follows.

\begin{definition}Let $X$ and $Y$ be Banach spaces, and let $1\leq\xi\leq\omega_1$ be an ordinal.  Fix any normalized basis $e=(e_n)$.  We define $\mathcal{JS}_{e,\xi}(X,Y)$ as the set of all operators $T\in\mathcal{L}(X,Y)$ such that for any normalized basic sequence $(x_n)$ in $X$ satisfying $(x_n)\leq(e_n)$, the image sequence $(Tx_n)$ fails to $\mathcal{S}_\xi$-dominate $(e_n)$.\end{definition}

\noindent So, we have weakened the definition of class $\mathcal{WS}_{e,\xi}$ by considering only those normalized basic sequences which are dominated by $(e_n)$.  This will ensure that we can get closure under addition in the non-Schreier cases, that is, for classes $\mathcal{JS}_{e,\omega_1}$.

Note that due to the strength of the $\ell_1$ norm, every normalized basic sequence is dominated by the canonical basis of $\ell_1$, and so in case $e=(e_n)$ is the canonical basis for $\ell_1$ we have
\[\mathcal{R}_\xi=\mathfrak{SM}_1^\xi=\mathcal{WS}_{e,\xi}=\mathcal{JS}_{e,\xi},\;\;\;1\leq\xi\leq\omega_1.\]
Here, $\mathcal{R}_{\omega_1}=\mathcal{R}$ denotes the Rosenthal operators, and $\mathcal{R}_\xi=\mathcal{WS}_{e,\xi}$, $1\leq\xi<\omega_1$, denotes the $\xi$th-order {\it Schreier Rosenthal} operators, defined in \cite{BF11}; classes $\mathfrak{SM}_1^\xi$, $1\leq\xi\leq\omega_1$, were defined in \cite[\S4]{BCFW15}.  Hence, each of these forms a norm-closed operator ideal by \cite[Theorem 4.3]{BCFW15}.  However, if $e=(e_n)$ is the canonical basis for $\ell_p$, $1<p<\infty$, or $c_0$, then we do not yet know whether $\mathcal{JS}_{e,\xi}$ is closed under addition for any $1\leq\xi<\omega_1$.

Let us now observe some straightforward consequences of the definition of $\mathcal{JS}_{e,\xi}$.

\begin{proposition}\label{property-list}Let $X$ and $Y$ be Banach spaces, let $1\leq\xi\leq\omega_1$, and let $e=(e_n)$ be the canonical basis for $\ell_p$, $1\leq p<\infty$, or $c_0$.
\begin{enumerate}\item\label{bounded-sequence}  An operator $T\in\mathcal{L}(X,Y)$ is class $\mathcal{JS}_{e,\xi}$ just in case for every bounded sequence $(x_n)$ in $X$ which is dominated by $(e_n)$, and every $\epsilon>0$, there exists $(a_n)\in c_{00}$ and $F\in\mathcal{S}_\xi$ such that $\norm{\sum a_nTx_n}<\epsilon\norm{\sum a_ne_n}$.
\item\label{JS-inclusion}  If $\zeta$ is another ordinal with $1\leq\xi\leq\zeta\leq\omega_1$, then
\[\mathcal{JS}_{e,\xi}(X,Y)\subseteq\mathcal{JS}_{e,\zeta}(X,Y).\]
\item\label{JS-compact}  Every compact operator in $\mathcal{L}(X,Y)$ is class $\mathcal{JS}_{e,\xi}$.  In other words,
\[\mathcal{K}(X,Y)\subseteq\mathcal{JS}_{e,\xi}(X,Y).\]
\item\label{JS-embedding}  Suppose $Z$ is also a Banach space, $T\in\mathcal{L}(X,Y)$ is an operator, and $J:Y\to Z$ is a continuous linear embedding.  If $T\notin\mathcal{JS}_{e,\xi}(X,Y)$ then $JT\notin\mathcal{JS}_{e,\xi}(X,Z)$.
\item\label{JS-closed}  $\mathcal{JS}_{e,\xi}(X,Y)$ is a norm-closed subset of $\mathcal{L}(X,Y)$.
\item\label{JS-multiplicative-ideal}  Suppose $W$ and $Z$ are also Banach spaces.  If $T\in\mathcal{JS}_{e,\xi}(X,Y)$, $A\in\mathcal{L}(W,X)$, and $B\in\mathcal{L}(Y,Z)$, then $BTA\in\mathcal{JS}_{e,\xi}(W,Z)$.\end{enumerate}\end{proposition}

\begin{proof}The proof of \eqref{bounded-sequence} is almost identical to the proof of Proposition \ref{seminormalized}, so we omit it.

\eqref{JS-inclusion} follows from the spreading property for Schreier families, together with the fact that for any pair of ordinals $1\leq\xi\leq\zeta\leq\omega_1$ there exists $d=d(\xi,\zeta)$ such that for any $S\in\mathcal{S}_\xi$ with $d\leq\min S$ we have $S\in\mathcal{S}_\zeta$.

\eqref{JS-compact} is just another consequence of Proposition \ref{WS-compact}.

\eqref{JS-embedding} is immediate from the definition of $\mathcal{JS}_{e,\xi}$.

Let us now prove \eqref{JS-closed}.  Suppose $(T_j)$ is a sequence in $\mathcal{JS}_{e,\xi}(X,Y)$ with $T_j\to T$ for some $T\in\mathcal{L}(X,Y)$.  Let $(x_n)$ be a normalized basic sequence in $X$ which is $C$-dominated, $C\in[1,\infty)$, by $(e_n)$, and let $\epsilon>0$.  Find $T_j$ with $\norm{T_j-T}<\frac{\epsilon}{2C}$.  Now let $(a_n)_{n\in F}$, $F\in\mathcal{S}_\xi$ be such that
\[\norm{\sum_{n\in F} a_nT_jx_n}<\frac{\epsilon}{2}\norm{\sum_{n\in F}a_ne_n}.\]
Then
\[\norm{\sum_{n\in F} a_nTx_n}\leq\norm{\sum_{n\in F} a_nT_jx_n}+\norm{T-T_j}\norm{\sum_{n\in F} a_nx_n}<\epsilon\norm{\sum_{n\in F}a_ne_n}.\]
This completes the proof of \eqref{JS-closed}.

Lastly, we shall prove \eqref{JS-multiplicative-ideal}.  Let $(w_n)$ be a bounded sequence in $W$ which is dominated by $(e_n)$, and let $\epsilon>0$.  Then $(Aw_n)$ is a bounded sequence in $X$ which is dominated by $(e_n)$, which means by \eqref{bounded-sequence} that we can find $F\in\mathcal{S}_\xi$ and $(a_n)_{n\in F}$ so that
\[\norm{\sum_{n\in F}a_nBTAw_n}\leq\norm{B}\norm{\sum_{n\in F}TAw_n}<\frac{\epsilon}{\norm{A}}\norm{\sum_{n\in F}Aw_n}\leq\epsilon\norm{\sum_{n\in F}Aw_n}.\]\end{proof}

\begin{remark}Observe that in the previous Proposition \ref{property-list}, the assumption $e=(e_n)$ is the canonical basis for $\ell_p$, $1\leq p<\infty$, or $c_0$, is only required for parts \eqref{bounded-sequence} and \eqref{JS-multiplicative-ideal}.  In parts \eqref{JS-inclusion}, \eqref{JS-compact}, \eqref{JS-embedding}, and \eqref{JS-closed}, $e=(e_n)$ could be any normalized basis for a Banach space $E$.\end{remark}

If we want classes $\mathcal{JS}_{e,\xi}$ form operator ideals, it remains to show that they are closed under addition.  In case $1\leq\xi<\omega_1$, this is unknown.  However, below we present a partial result in that direction, which turns out to be sufficient for our purposes here.

\begin{proposition}\label{JS-Schreier-addition}Let $X$ and $Y$ be Banach spaces, let $e=(e_n)$ denote the canonical basis for $\ell_p$, $1\leq p<\infty$, or $c_0$, and let $1\leq\xi,\zeta\leq\omega_1$ be ordinals.  Suppose $S\in\mathcal{JS}_{e,\xi}(X,Y)$ and $T\in\mathcal{JS}_{e,\zeta}(X,Y)$.
\begin{itemize}\item[(i)]  If $\xi=\omega_1$ or $\zeta=\omega_1$ then $S+T\in\mathcal{JS}_{e,\omega_1}(X,Y)$.
\item[(ii)]  If $\xi$ and $\zeta$ are both countable, i.e. $<\omega_1$, then $S+T\in\mathcal{JS}_{e,\xi+\zeta}(X,Y)$.\end{itemize}\end{proposition}

\begin{proof}Pick any bounded sequence $(x_n)$ in $X$ which is dominated by $(e_n)$.  By Proposition \ref{1.iii} we can find $L=(n_k)\in[\mathbb{N}]$ such that $\mathcal{S}_\xi[\mathcal{S}_\zeta](L)\subseteq\mathcal{S}_{\zeta+\xi}$ in case (ii), and let $L=(n_k)=\mathbb{N}$ in case (i).  Now, by successively considering the tails of $(x_{n_k})$ and using the spreading property of $\mathcal{S}_\zeta$ we can find $F_1<F_2<F_3<\cdots\in\mathcal{S}_\zeta$ and scalars $(a_k)$ such that
\[\norm{\sum_{k\in F_j}a_kTx_{n_k}}<\epsilon 2^{-j-1}\;\;\;\text{ and }\;\;\;\norm{\sum_{k\in F_j}a_ke_k}=1\;\;\;\text{ for all }j\in\mathbb{N}.\]
Let us form matching block sequences by setting
\[x'_j=\sum_{k\in F_j}a_kx_{n_k},\;\;\;\text{ and }\;\;\;e'_j=\sum_{k\in F_j}a_ke_k,\;\;\;j\in\mathbb{N}.\]
Recall that every normalized block sequence of $(e_n)$ is 1-equivalent to $(e_n)$ (cf., e.g., \cite[Lemma 2.1.1]{AK06}).  In particular, $(e'_j)$ is 1-equivalent to $(e_n)$, which means $(x'_j)$ is dominated by $(e_n)$.  We can therefore find scalars $(b_j)\in c_{00}$ and $F\in\mathcal{S}_\xi$ such that
\[\norm{\sum_{j\in F}b_jSx'_j}<\frac{\epsilon}{2}\;\;\;\text{ and }\;\;\;\norm{\sum_{j\in F}b_je'_j}=\norm{\sum_{j\in F}b_je_j}=1.\]
Next, due to H\"older's inequality, we obtain
\[\norm{\sum_{j\in F}b_j(S+T)x'_j}<\frac{\epsilon}{2}+\frac{\epsilon}{2}\sum_{j\in F}|b_j|2^{-j}\leq\epsilon=\epsilon\norm{\sum_{j\in F}b_je'_j}.\]
In case (i) we are already done, and in case (ii) we need only recall that $\mathcal{S}_\xi[\mathcal{S}_\zeta](L)\subseteq\mathcal{S}_{\zeta+\xi}$, so that we are done anyway.\end{proof}

The limitations on the above proposition prevent us from concluding that $\mathcal{JS}_{e,\xi}$ is an operator ideal when $1\leq\xi<\omega_1$.  However, if we combine Proposition \ref{property-list}\eqref{JS-compact},\eqref{JS-closed},\eqref{JS-multiplicative-ideal}, and Proposition \ref{JS-Schreier-addition}, we obtain the following nice result when $\xi=\omega_1$.

\begin{theorem}\label{JS-ideal}Let $e=(e_n)$ denote the canonical basis for $\ell_p$, $1\leq p<\infty$, or $c_0$.  Then $\mathcal{JS}_{e,\omega_1}$ is a norm-closed operator ideal.\end{theorem}

We also have the following property for this same special case $\xi=\omega_1$.

\begin{proposition}\label{contains-lp}Let $e=(e_n)$ be a normalized basis for a Banach space $E$, and let $X$ and $Y$ be Banach spaces such that either $X$ or $Y$ fails to contain a copy of $E$.  Then
\[\mathcal{JS}_{e,\omega_1}(X,Y)=\mathcal{L}(X,Y).\]\end{proposition}

\begin{proof}Indeed, suppose $\mathcal{JS}_{e,\xi}(X,Y)\neq\mathcal{L}(X,Y)$.  Then there is a linear operator $T\in\mathcal{L}(X,Y)\setminus\mathcal{JS}_{e,\xi}(X,Y)$, which means we can find a normalized basic sequence $(x_n)$ in $X$ which is dominated by $(e_n)$, and such that $(Tx_n)$ dominates $(e_n)$.  Then
\[(x_n)\leq(e_n)\leq(Tx_n)\leq(x_n)\]
so that $X$ and $Y$ both contain copies of $E$.\end{proof}

Recall that every seminormalized basic sequence in a reflexive space is weakly null.  Together with the Bessaga-Pe\l czy\'{n}ski Selection Principle and the Principle of Small Perturbations, this means that any seminormalized basic sequence in a reflexive space with a basis admits a subsequence equivalent to a normalized block basic sequence.  In the case of $\ell_q$, $1<q<\infty$, this is in turn equivalent to the canonical basis (cf., e.g., \cite[Lemma 2.1.1]{AK06}), so that every normalized basic sequence has a subsequence dominated by the canonical basis $(e_n)$ of $\ell_p$, $1\leq p\leq q$.  Obviously, if $q=1$ and $1\leq p\leq q$ then $p=1$ so that, again, every normalized basic sequence in $\ell_q$ is dominated by $(e_n)$.  From this we obtain the following.

\begin{proposition}\label{JS-WS}Let $Y$ be a Banach space, let $1\leq\xi\leq\omega_1$ be an ordinal, and let $e=(e_n)$ denote the canonical basis for $\ell_p$, where $1\leq p\leq q<\infty$.  Then
\[\mathcal{JS}_{e,\xi}(\ell_q,Y)=\mathcal{WS}_{e,\xi}(\ell_q,Y).\]\end{proposition}

\noindent Consequently, in these cases we can apply a nice result from \cite{BF11}.

\begin{theorem}[{\cite[Corollary 19]{BF11}}]\label{corollary-19}Let $X$ and $Y$ be separable Banach spaces, and let $e=(e_n)$ denote any normalized 1-spreading basis.  (In particular, $e=(e_n)$ can be chosen from among the canonical bases for $\ell_p$, $1\leq p<\infty$, or $c_0$.)  If $T\in\mathcal{WS}_{e,\omega_1}(X,Y)$ then there exists a countable ordinal $1\leq\xi<\omega_1$ such that $T\in\mathcal{WS}_{e,\xi}(X,Y)$.\end{theorem}

\noindent This shall be used to prove the following.

\begin{theorem}\label{lp-linfty}Let $1\leq p<\infty$, and let $Z=\ell_{p'}$ with $\frac{1}{p}+\frac{1}{p'}=1$ if $p\neq 1$ and $Z=c_0$ if $p=1$.  Then $\mathcal{L}(\ell_1,Z)$ and $\mathcal{L}(\ell_p,\ell_\infty)$ each admit an uncountable chain of strictly increasing closed subideals.\end{theorem}

\begin{proof}Denote by $e=(e_n)$ the canonical basis for $\ell_p$.  For a countable ordinal $1\leq\xi<\omega_1$, let us set
\[\mathcal{J}_\xi=\overline{\bigcup_{k=1}^\infty\mathcal{JS}_{e,\xi k}(\ell_p,\ell_\infty)}.\]
It is clear from Proposition \ref{property-list}\eqref{JS-inclusion},\eqref{JS-closed},\eqref{JS-multiplicative-ideal}, and Proposition \ref{JS-Schreier-addition}, that if $\zeta:=\sup_k\xi k$ then $\mathcal{J}_\xi$ is a closed subideal of $\mathcal{L}(\ell_p,\ell_\infty)$ contained in $\mathcal{JS}_{e,\zeta}(\ell_p,\ell_\infty)$.  Hence
\[\mathcal{J}_\xi^*=\left\{T\in\mathcal{L}(\ell_1,Z):T^*\in\mathcal{J}_\xi\right\}.\]
is a closed subideal of $\mathcal{L}(\ell_1,Z)$.

We claim that there is a countable ordinal $\alpha>\zeta$ and an operator $A_\alpha\in\mathcal{L}(\ell_1,Z)$ such that $A_\alpha^*\in\mathcal{J}_\alpha\setminus\mathcal{J}_\xi$.  For proof of the claim, let $T_\zeta^p$ denote the $p$-convexification of the Tsirelson space of order $\zeta$.  It is well-known that $T_\zeta^p$ is a reflexive space containing no copy of $\ell_p$, with a normalized basis $(t_n)$ which is dominated by $(e_n)$, and which $\mathcal{S}_\zeta$-dominates $(e_n)$ (cf., e.g., \cite[\S3, p88]{Wa14}).  Thus, there exists an operator $I_\zeta\in\mathcal{L}(T_\zeta^{p*},Z)$ whose dual is the continuous formal inclusion of $j_\zeta:\ell_p\to T_\zeta^p$ which maps $e_n\mapsto t_n$. Indeed, in the case $1<p<\infty$, due to reflexivity of both spaces we can simply set $I_\zeta = (j_\zeta)^*$. In case $p=1$, due to the weakness of the $c_0$ norm we can straightforwardly define $I_\zeta\in\mathcal{L}(T_\zeta^{1*},c_0)$ by the rule $I_\zeta t_n^*=f_n$, where $(f_n)$ is the canonical basis of $c_0$ and $(t_n^*)\subset T_\zeta^{1*}$ are the biorthogonal functionals to $(t_n)\subset T_\zeta^1$. 
 %In this case, since $(e_n)$ and $(t_n)$ are respectively biorthogonal to $(f_n)$ and $(t_n^*)$, it is easy to see that $I_\zeta^*\in\mathcal{L}(\ell_1,T_\zeta^1)$ is again the desired formal inclusion $I_\zeta^*e_n=t_n$.  
 Recall that every separable Banach space is isometrically isomorphic to a quotient of $\ell_1$ (cf., e.g., \cite[Corollary 2.3.2]{AK06}).  Thus, there exists a surjection $Q_\zeta:\ell_1\to T_\zeta^{p*}$.  Recall that embeddings and surjections are dual sorts of operators (cf., e.g., \cite[Lemma 1.30]{Ai04}) so that $Q_\zeta^*:T_\zeta^p\to\ell_\infty$ is an embedding.  Since $I_\zeta^*$ is not class $\mathcal{JS}_{e,\zeta}$, by Proposition \ref{JS-embedding} neither is $Q_\zeta^*I_\zeta^*$.  On the other hand, since $T_\zeta^p$ fails to contain a copy of $\ell_p$, by Propositions \ref{contains-lp} and \ref{JS-WS} we have $I_\zeta^*\in\mathcal{WS}_{e,\omega_1}(\ell_p,T_\zeta^p)$.  Now we can apply Theorem \ref{corollary-19} to find a countable ordinal $1\leq\alpha<\omega_1$ such that $I_\zeta^*\in\mathcal{WS}_{e,\alpha}(\ell_p,T_\zeta^p)$.  It then follows from Propositions \ref{property-list}\eqref{JS-multiplicative-ideal} and \ref{JS-WS} that $Q_\zeta^*I_\zeta^*\in\mathcal{JS}_{e,\alpha}(\ell_p,\ell_\infty)$.  Due to Proposition \ref{property-list}\eqref{JS-inclusion}, this forces $\alpha>\zeta$.  Letting $A_\alpha=I_\zeta Q_\zeta$ completes the proof of the claim.

Due to the claim above together with Proposition \ref{property-list}\eqref{JS-inclusion}, for any countable ordinal $1\leq\xi<\omega_1$ there exists a countable ordinal $\alpha>\xi$ such that $\mathcal{J}_\xi\subsetneq\mathcal{J}_\alpha$ and $\mathcal{J}_\xi^*\subsetneq\mathcal{J}_\alpha^*$.  The desired chains follow from the fact that if $(\xi_i)_{i\in I}$ is any countable chain of ordinals then $\sup_i\xi_i$ is again a countable ordinal.\end{proof}

We can really just deduce the main Theorem \ref{main} of this section as a corollary to the above, in light of the following elementary fact.

\begin{proposition}\label{complemented-injection}Let $X$ and $Y$ be Banach spaces, and let $Z$ be a Banach space containing complemented copies of $X$ and $Y$.  For each closed subideal $\mathcal{I}$ in $\mathcal{L}(X,Y)$, we define
\[\Psi(\mathcal{I}):=[\mathcal{G}_\mathcal{I}](Z),\]
the closed linear span in $\mathcal{L}(Z)$ of operators factoring through elements of $\mathcal{I}$.  Then $\Psi$ an order-preserving injection from the closed subideals of $\mathcal{L}(X,Y)$ into the closed ideals of $\mathcal{L}(Z)$.\end{proposition}

\begin{proof}This is just a minor adaptation of \cite[Proposition 3.9]{Wa15}.  Let $\mathcal{I}$ and $\mathcal{J}$ be closed subideals in $\mathcal{L}(X,Y)$.  Clearly, if $\mathcal{I}\subseteq\mathcal{J}$, then $\Psi(\mathcal{I})\subseteq\Psi(\mathcal{J})$.  Now let us suppose instead that $\Psi(\mathcal{I})\subseteq\Psi(\mathcal{J})$, and pick any $T\in\mathcal{I}$.  Let $P:Z\to\widehat{X}$ and $R:Z\to\widehat{Y}$ denote projections onto subspaces $\widehat{X}$ and $\widehat{Y}$, respectively, such that there exist isomorphisms $U:X\to\widehat{X}$ and $V:Y\to\widehat{Y}$.  Also, let $J:\widehat{X}\to Z$ and $Q:\widehat{Y}\to Z$ denote the corresponding embeddings, i.e. such that $PJ$ and $RQ$ are just identity operators acting on $\widehat{X}$ and $\widehat{Y}$, respectively.  Then $QVTU^{-1}P\in\Psi(\mathcal{I})\subseteq\Psi(\mathcal{J})$, and so we can find a sequence of finite sums satisfying
\[\lim_{n\to\infty}\sum_{j=1}^{m_n}B_{n,j}T_{n,j}A_{n,j}=QVTU^{-1}P,\]
where $A_{n,j}\in\mathcal{L}(Z,X)$, $B_{n,j}\in\mathcal{L}(Y,Z)$, and $T_{n,j}\in\mathcal{J}$ for all $n$ and $j$.  Let us set
\[S_n:=\sum_{j=1}^{m_n}V^{-1}RB_{n,j}T_{n,j}A_{n,j}JU\in\mathcal{J}.\]
Then $S_n\to V^{-1}RQVTU^{-1}PJU=T$, and since $\mathcal{J}$ is closed we get $T\in\mathcal{J}$.  This shows that $\mathcal{I}\subseteq\mathcal{J}$ as desired.\end{proof}

\end{document}